\documentclass[12pt]{amsart}
\usepackage{hyperref}
\usepackage[utf8]{inputenc}
\usepackage{amssymb, graphicx}
\usepackage[notref,notcite]{showkeys}
\usepackage{amsfonts}
\usepackage{amsmath}
\usepackage{latexsym}
\usepackage{amscd}
\usepackage{verbatim}
\usepackage{mathrsfs}

\allowdisplaybreaks[4]

\newcommand{\mf}{\mathfrak}
\newcommand{\g}{\mf{g}}

\newcommand{\gl}{\mf{gl}}

\allowdisplaybreaks[4]

\newcommand{\Z}{{\mathbb Z}}
\newcommand{\C}{{\mathbb C}}
\newcommand{\N}{{\mathbb N}}
\newcommand{\Q}{{\mathbb Q}}

\newcommand{\K}{{\mathcal K}}

\newcommand{\supp}{{\operatorname{Supp}}\xspace}

\renewcommand{\phi}{\varphi}

\def\sl{\mathfrak{sl}}
\def\gl{\mathfrak{gl}}

\def\l{\lambda}

 \def\ann{\rm ann}
\def\sl{\mathfrak{sl}}

\newtheorem{theorem}{Theorem}[section]

\newtheorem{lemma}[theorem]{Lemma}

\theoremstyle{remark}

\numberwithin{equation}{section}

\def\span{\mathrm{span}}
\def\supp{\mathrm{Supp}}
\def\vn{\varnothing}
\def\m{\mathfrak{m}}

\newcommand\pxi[1]{\frac{\partial}{\partial\xi_{#1}}}

\begin{document}




\title[]{bounded weight modules over the Lie superalgebra of Cartan W-type}
\author{Rencai L\"{u},  Yaohui Xue}
\maketitle

\begin{abstract}
  Let $A_{m,n}$ be the tensor product of the polynomial algebra in $m$ even variables and the exterior algebra in $n$ odd variables over the complex field $\C$, and the Witt superalgebra $W_{m,n}$ be the Lie superalgebra of superderivations of $A_{m,n}$. In this paper, we classify the non-trivial simple bounded weight $W_{m,n}$ modules with respect to the standard Cartan algebra of $W_{m,n}$. Any such module is a simple quotient of a tensor module $F(P,L(V_1\otimes V_2))$ for a simple weight module $P$ over the Weyl superalgebra $\mathcal K_{m,n}$, a finite-dimensional simple $\gl_m$-module $V_1$ and a simple bounded $\gl_n$-module $V_2$.
\end{abstract}




\section{Introduction}

We denote by $\Z, \Z_+, \N, \Q$ and $\C$ the sets of all integers, non-negative integers, positive integers, rational numbers and complex numbers, respectively. All vector spaces and algebras in this paper are over $\C$. Any module over a Lie superalgebra or an associative superalgebra is assumed to be $\Z_2$-graded. Let $0\neq (m,n)\in\Z_+^2$ and let $e_1,\dots,e_{m+n}$ be the standard basis of $\C^{m+n}$.

Let $A_{m,n}$ ( resp. $\mathcal A_{m,n}$ ) be the tensor superalgebra of the polynomial algebra $\C[t_1,\dots,t_m]$ ( resp. Laurent polynomial algebra $\C[t_1^{\pm 1},\dots,t_m^{\pm 1}]$ ) in $m$ even variables $t_1,\dots,t_m$ and the exterior algebra $\Lambda(n)$ in $n$ odd variables $\xi_1,\dots,\xi_n$. Denote by $W_{m,n}$ ( resp. $\mathcal W_{m,n}$ ) the Lie superalgebra of super-derivations of $A_{m,n}$ ( resp. $\mathcal A_{m,n}$ ). Cartan $W$-type Lie superalgebra $W_{m,n}$ was introduced by V. Kac in \cite{K}.

Weight modules with finite-dimensional weight spaces are called Harish-Chandra modules. Many efforts have been made towards the classification of Harish-Chandra modules over various Lie (super)algebras. For finite-dimensional simple Lie algebras, O. Mathieu classified all the simple Harish-Chandra modules in \cite{Ma1}. M. Gorelik and D. Grantcharov completed the classification of all simple Harish-Chandra modules over all classical Lie superalgebras in \cite{GG1}, following the works in \cite{DMP,FGG,GG, Gr,Ho}. Such modules over the Virasoro algebra (which is the universal central extension of $\mathcal W_{1,0}$) were conjectured by V. Kac and classified by O. Mathieu in \cite{Ma}. Y. Billig and V. Futorny completed the classification for $\mathcal W_{m,0}$ in \cite{BF1}. Simple Harish-Chandra modules over $W_{0,n}$ were classified in \cite{DMP}. The simple weight modules with finite-dimensional weight spaces with respect to the Cartan subalgebra of $\mathcal W_{1,0}$ over the $N=2$ Ramond algebra (which is a central extension of $\mathcal W_{1,1}$) were classified in \cite{Liu1}. Such modules over $\mathcal W_{m,n}$ were classified in \cite{XL2}, see also \cite{BFIK}.
For more related results, we refer the readers to \cite{ BF2, BL,CLL, E1, E2, LZ2, MZ, Sh, Su1, Su2} and the references therein.

As we know, the classification of simple bounded weight modules is not only  is an important step in the classification of simple Harish-Chandra modules but also interesting on its own.  Simple bounded weight modules over $W_{1,0}$ were classified in \cite{Ma} and simple bounded modules over $W_{2,0}$ were classified in \cite{CG}. Such modules over $W_{m,0}$ were classified in \cite{XL1}. In this paper, we classify the non-trivial simple bounded weight $W_{m,n}$ modules with respect to the standard Cartan algebra of $W_{m,n}$. Any such module is a simple quotient of a tensor module $F(P,L(V_1\otimes V_2))$ for a simple weight module $P$ over the Weyl superalgebra $\mathcal K_{m,n}$, a finite-dimensional simple $\gl_m$-module $V_1$ and a simple bounded $\gl_n$-module $V_2$. This is achieved by widely using the results and methods in \cite{XL1, XL2}.

This paper is arranged as follows. In Section 2, we give some definitions and preliminaries. In Section 3, we prove that a simple weight $AW$-module with a finite-dimensional weight space is a tensor module, see Lemma \ref{cusF(P,M)} and Theorem \ref{F(P,M)}. In Section 4, we prove our main theorem, see Theorem \ref{main}.

\section{Preliminaries}

A vector space $V$ is called a superspace if $V$ is endowed with a $\Z_2$-gradation $V=V_{\bar 0}\oplus V_{\bar 1}$. For any homogeneous element $v\in V$, let $|v|\in\Z_2$ with $v\in V_{|v|}$. Throughout this paper, $v$ is always assumed to be a homogeneous element whenever we write $|v|$ for a vector $v\in V$.

A module over a Lie superalgebra or an associative superalgebra is simple if it does not have nontrivial $\Z_2$-graded submodules. A module $M$ over a Lie superalgebra or an associative superalgebra $\g$ is called strictly simple if $M$ does not have $\g$-invariant subspaces except $0$ and $M$. Clearly, a strictly simple module must be simple. Denote by $\Pi(M)$ the parity-change of $M$ for a module $M$ over a Lie superalgebra or an associative superalgebra.

\begin{lemma}\cite[Lemma 2.2]{XL2}\label{density}
Let $B,B'$ be two unital associative superalgebras such that $B'$ has a countable basis, $R=B\otimes B'$. Then
\begin{itemize}
  \item[(1)]Let $M$ be a $B$-module and $M'$ be a strictly simple $B'$-module. Then $M\otimes M'$ is a simple $R$-module if and only if $M$ is simple.
  \item[(2)]Suppose that $V$ is a simple $R$-module and $V$ contains a strictly simple $B'=\C\otimes B'$-submodule $M'$. Then $V\cong M\otimes M'$ for some simple $B$-module $M$.
\end{itemize}
\end{lemma}

Write $A:=A_{m,n}$, $W:=W_{m,n}$ and omit $\otimes$ in $A$ for convenience.

For any $\alpha=(\alpha_1,\dots,\alpha_m)\in\Z_+^m$ and $i_1,\dots,i_k\in\{1,\dots,n\}$, write $t^\alpha:=t_1^{\alpha_1}\cdots t_m^{\alpha_m}$ and $\xi_{i_1,\dots,i_k}:=\xi_{i_1}\cdots \xi_{i_k}$. Also, for any subset $I=\{i_1,\dots,i_k\} \subset \{1,\dots,n\}$, write $\underline{I}=(l_1,\dots,l_k)$ if $\{l_1,\dots,l_k\}=\{i_1,\dots,i_k\}$ and $l_1<\dots<l_k$. Denote $\xi_I:=\xi_{l_1,\dots,l_k}$ and set $\xi_\varnothing=1$.

Let $i_1,\dots,i_k$ be a sequence in $\{1,\dots,n\}$. Denote by $\tau(i_1,\dots,i_k)$ the inverse order of the sequence $i_1,\dots,i_k$. Let $I,J\subset \{1,\dots,n\}$ with $I\cap J=\varnothing$ and $\underline{I}=(k_1,\dots,k_p),\underline J=(l_1,\dots,l_q)$. We write $\tau(I,J)=(k_1,\dots,k_p,l_1,\dots,l_q)$. Set $\tau(\varnothing,\varnothing)=0$. Then $\xi_{I\cup J}=(-1)^{\tau(I,J)}\xi_I\xi_J$ for all $I\cap J=\varnothing$.

$W$ has a standard basis
$$\{t^\alpha\xi_I\frac{\partial}{\partial t_i},t^\alpha\xi_I\frac{\partial}{\partial\xi_j}\ |\ \alpha\in\Z_+^m,I\subset\{1,\dots,n\},i\in\{1,\dots,m\},j\in\{1,\dots,n\}\}.$$

Define the extended Witt superalgebra $\tilde{W}=W\ltimes A$ by
\begin{equation}
[a,a']=0,[x,a]=-(-1)^{|x||a|}[a,x]=x(a),\forall x\in W,a,a'\in A.
\end{equation}

Write $d_i:=t_i\frac{\partial}{\partial t_i}$ for any $i\in\{1,\dots,m\}$ and $\delta_j:=\xi_j\frac{\partial}{\partial\xi_j}$ for any $j\in\{1,\dots,n\}$. Then $H=\span \{d_i,\delta_j\ |\ i=1,\dots,m,j=1,\dots,n\}$ is the standard Cartan subalgebra of $W$. Let $\g$ be any Lie super-subalgebra of $\tilde W$ that contains $H$ and let $M$ be a $\g$-module. $M$ is called a weight module if the action of $H$ on $M$ is diagonalizable. Namely, $M$ is a weight module if $M=\oplus_{\lambda\in\C^m,\mu\in\C^n}M_{(\lambda,\mu)}$, where
$$M_{(\lambda,\mu)}=\{v\in M\ |\ d_i(v)=\lambda_iv,\delta_j(v)=\mu_jv,i\in\{1,\dots,m\},j\in\{1,\dots,n\}\}$$
is called the weight space with weight $(\lambda,\mu)$.
Denote by
$$\supp(M)=\{(\lambda,\mu)\in\C^{m+n}\ |\ M_{(\lambda,\mu)}\neq 0\}$$
the support set of $M$. A weight $\g$-module is called if the dimensions of its weight spaces are uniformly bounded by a constant positive integer.

It's easy to see that $\tilde{W}$ (resp. $W$) itself is a weight module over $\tilde{W}$ (resp. $W$) with $\supp(\tilde{W})$ (resp. $W$) $\subset \Z^{m+n}$. So for any indecomposable weight module $M$ over $\tilde{W}$ or $W$ we have $\supp(M)\in(\lambda,\mu)+\Z^{m+n}$ with some $(\lambda,\mu)\in\C^{m+n}$.

Let $\gl(m,n)=\gl(\C^{m|n})$ be the general linear Lie superalgebra realized as the spaces of all $(m+n)\times(m+n)$ matrices. Denote by $E_{i,j}, i,j=1,2,\ldots,m+n$ be the $(i,j)$-th matrix unit.
$\gl(m,n)$ has a $\Z$-gradation $\gl(m,n)=\gl(m,n)_{-1}\oplus\gl(m,n)_0\oplus\gl(m,n)_1$, where
\begin{eqnarray*}
&\gl(m,n)_{-1}=\span\{E_{m+j,i}\ |\ i\in\{1,\dots,m\},j\in\{1,\dots,n\}\},\\
&\gl(m,n)_1=\span\{E_{i,m+j}\ |\ i\in\{1,\dots,m\},j\in\{1,\dots,n\}\}
\end{eqnarray*}
and $\gl(m,n)_0=\gl(m,n)_{\bar 0}$. Obviously, this $\Z$-gradation is consistent with the $\Z_2$-gradation of $\gl(m,n)$.

A $\gl(m,n)$-module $M$ is a weight module if $M=\oplus_{\lambda\in\C^m,\mu\in\C^n}M_{(\lambda,\mu)}$, where
$$M_{(\lambda,\mu)}=\{v\in M\ |\ E_{i,i}(v)=\lambda_iv,E_{m+j,m+j}(v)=\mu_jv,i\in\{1,\dots,m\},j\in\{1,\dots,n\}\}$$
is called the weight space with weight $(\lambda,\mu)$.
Denote by
$$\supp(M)=\{(\lambda,\mu)\in\C^{m+n}\ |\ M_{(\lambda,\mu)}\neq 0\}$$
the support set of $M$.
A weight $\gl(m,n)$-module is called if the dimensions of its weight spaces are uniformly bounded by a constant positive integer.

For any module V over the Lie algebra $\gl(m,n)_0$, $V$ could be viewed as modules over the Lie superalgebra $\gl(m,n)_0$ with $V_{\bar 0}=V$.
Let $V$ be a $\gl(m,n)_0$-module and extend $V$ trivially to a $\gl(m,n)_0\oplus\gl(m,n)_1$-module.
The \emph{Kac module} of $V$ is the induced module $K(V):=\rm {Ind}_{\gl(m,n)_0\oplus\gl(m,n)_1}^{\gl(m,n)}(V)$. It's easy to see that $K(V)$ is isomorphic to $\Lambda(\gl(m,n)_{-1})\otimes V$ as superspaces.

\begin{lemma}\cite[Theorem 4.1]{CM}\label{L(V)}
For any simple $\gl(m,n)_0$-module $V$, the module $K(V)$ has a unique maximal submodule. The unique simple top of $K(V)$ is denoted $L(V)$. Any simple $\gl(m,n)$-module is isomorphic to $L(V)$ for some simple $\gl(m,n)_0$-module $V$ up to a parity-change.
\end{lemma}

Clearly, $L(V)$ is a weight $\gl(m,n)$-module if and only if $V$ is a weight $\gl(m,n)_0$-module.

\section{$AW$-modules}

A $\tilde{W}$-module $M$ is called an $AW$-module if the action of $A$ on $M$ is associative, i.e.,
$$a\cdot a'\cdot v=(aa')\cdot v,t^0\cdot v=v,\forall a,a'\in A,v\in M.$$

In this section, we will classify all simple bounded $AW$-modules.

For any Lie (super)algebra $\g$, let $U(\g)$ be the universal enveloping algebra of $\g$. By the PBW Theorem, $U(\tilde W)=U(A)\cdot U(W)$. Let $\mathcal{J}$ be the left ideal of $U(\tilde W)$ generated by
$$\{t^0-1,t^\alpha\xi_I\cdot t^\beta\xi_J-t^{\alpha+\beta}\xi_I\xi_J\ |\ \alpha,\beta\in\Z_+^m, I,J\subset \{1,\dots,n\}\}.$$
It is easy to see that $\mathcal{J}$ is an ideal of $U(\tilde W)$. Let $\bar{U}$ be the quotient algebra $U(\tilde W)/\mathcal{J}$. Identify $A$ and $W$ with their images in $\bar U$, then $\bar U=A\cdot U(W)$.

$A\cdot W$ is a Lie super-subalgebra of $\bar U$, of which the bracket is given by
\begin{equation}
[a\cdot x,b\cdot y]=ax(b)\cdot y-(-1)^{|a\cdot x||b\cdot y|}by(a)\cdot x+(-1)^{|x||b|}ab\cdot [x,y],\forall a,b\in A,x,y\in W.
\end{equation}

For any $\alpha\in\Z_+^m,I\subset\{1,\dots,n\},\partial\in\{\frac{\partial}{\partial t_1},\dots,\frac{\partial}{\partial t_m},\pxi{1},\dots,\pxi{n}\}$, define
$$X_{\alpha,I,\partial}=\sum_{\substack{0\leqslant \beta\leqslant \alpha\\J\subset I}}(-1)^{|\beta|+|J|+\tau(J,I\setminus J)}\binom{\alpha}{\beta}t^\beta\xi_J\cdot t^{\alpha-\beta}\xi_{I\setminus J}\partial.$$
Then $X_{0,\vn,\partial}=\partial$.
Let
$$T=\span\{X_{\alpha,I,\partial}\ |\ \alpha\in\Z_+^m,I\subset\{1,\dots,n\},|\alpha|+|I|>0,\partial\in\{\frac{\partial}{\partial t_1},\dots,\frac{\partial}{\partial t_m},\pxi{1},\dots,\pxi{n}\}\}$$
and
$$\Delta=\span\{\frac{\partial}{\partial t_1},\dots,\frac{\partial}{\partial t_m},\pxi{1},\dots,\pxi{n}\}.$$

\begin{lemma}\label{compute}
\begin{itemize}
  \item[(1)]$[T,A]=[T,\Delta]=0.$
  \item[(2)]\begin{eqnarray*}
  &t^\alpha\xi_I\partial=\sum_{\substack{0\leqslant \beta\leqslant \alpha\\J\subset I}}(-1)^{\tau(J,I\setminus J)}\binom{\alpha}{\beta}t^\beta\xi_J\cdot X_{\alpha-\beta,I\setminus J,\partial},\\
  &\forall \alpha\in\Z_+^m,I\subset\{1,\dots,n\},\partial\in\{\frac{\partial}{\partial t_1},\dots,\frac{\partial}{\partial t_m},\pxi{1},\dots,\pxi{n}\}.
  \end{eqnarray*}
\end{itemize}
\end{lemma}
\begin{proof}
(1)Let $\alpha\in\Z_+^m,I\subset\{1,\dots,n\},\partial\in\{\frac{\partial}{\partial t_1},\dots,\frac{\partial}{\partial t_m},\pxi{1},\dots,\pxi{n}\}$. For any $i\in\{1,\dots,m\}$, we have
\begin{eqnarray*}
[\frac{\partial}{\partial t_i},X_{\alpha,I,\partial}]&=&\sum_{\substack{0\leqslant\beta\leqslant\alpha\\J\subset I}}(-1)^{|\beta|+|J|+\tau(J,I\setminus J)}\binom{\alpha}{\beta}\beta_it^{\beta-e_i}\xi_J\cdot t^{\alpha-\beta}\xi_{I\setminus J}\partial\\
&&+\sum_{\substack{0\leqslant\beta\leqslant\alpha\\J\subset I}}(-1)^{|\beta|+|J|+\tau(J,I\setminus J)}\binom{\alpha}{\beta}(\alpha-\beta)_it^\beta\xi_J\cdot t^{\alpha-\beta-e_i}\xi_{I\setminus J}\partial\\
&=&\sum_{\substack{0\leqslant\beta\leqslant\alpha\\J\subset I}}(-1)^{|\beta|+|J|+\tau(J,I\setminus J)}\binom{\alpha}{\beta}\beta_it^{\beta-e_i}\xi_J\cdot t^{\alpha-\beta}\xi_{I\setminus J}\partial\\
&&-\sum_{\substack{e_i\leqslant\beta\leqslant\alpha+e_i\\J\subset I}}(-1)^{|\beta|+|J|+\tau(J,I\setminus J)}\binom{\alpha}{\beta-e_i}(\alpha-\beta+e_i)_it^{\beta-e_i}\xi_J\cdot t^{\alpha-\beta}\xi_{I\setminus J}\partial\\
&=&0.
\end{eqnarray*}

Clearly, for any $j\in\{1,\dots,n\}\setminus I$, $[\pxi{j},X_{\alpha,I,\partial}]=0$. Now let $\underline I=(l_1,\dots,l_k)$ and $s\in\{1,\dots,k\}$.

For any $J\subset I$ with $l_s\in J$, let $\underline J=(l_{i_1},\dots,l_{i_p})$ and $s=i_q$ for some $q\in\{1,\dots,p\}$. Then there are $(s-1)-(q-1)=s-q$ elements in $I\setminus J$ that are less than $l_s$. It follows that $\tau(J,I\setminus J)=s-q+\tau(J\setminus\{l_s\},I\setminus J)$. So
\begin{eqnarray*}
(-1)^{|J|+\tau(J,I\setminus J)}\pxi{l_s}(\xi_J)&=&(-1)^{|J|+s-q+\tau(J\setminus\{l_s\},I\setminus J)}(-1)^{q-1}\xi_{J\setminus\{l_s\}}\\
&=&(-1)^{s+|J\setminus\{l_s\}|+\tau(J\setminus\{l_s\},I\setminus J)}\xi_{J\setminus\{l_s\}}.
\end{eqnarray*}

For any $J\subset I$ with $l_s\notin J$, let $\underline{I\setminus J}=(l_{i_{p+1}},\dots,l_{i_k})$ and $s=i_q$ for some $q\in\{p+1,\dots,k\}$. Then there are $(k-s)-(k-q)=q-s$ elements in $J$ that are greater than $l_s$. It follows that $\tau(J,I\setminus J)=q-s+\tau(J,I\setminus(J\cup\{l_s\}))$. So
\begin{eqnarray*}
(-1)^{|J|}(-1)^{|J|+\tau(J,I\setminus J)}\pxi{l_s}(\xi_{I\setminus J})&=&(-1)^{q-s+\tau(J,I\setminus(J\cup\{l_s\}))}(-1)^{q-1-|J|}\xi_{I\setminus(J\cup\{l_s\})}\\
&=&(-1)^{s+1+|J|+\tau(J,I\setminus(J\cup\{l_s\}))}\xi_{I\setminus(J\cup\{l_s\})}.
\end{eqnarray*}

Thus,
\begin{eqnarray*}
[\pxi{l_s},X_{\alpha,I,\partial}]&=&\sum_{\substack{0\leqslant\beta\leqslant\alpha\\J\subset I}}(-1)^{|\beta|+|J|+\tau(J,I\setminus J)}\binom{\alpha}{\beta}t^\beta\pxi{l_s}(\xi_J)\cdot t^{\alpha-\beta}\xi_{I\setminus J}\partial\\
&&+\sum_{\substack{0\leqslant\beta\leqslant\alpha\\J\subset I}}(-1)^{|J|}(-1)^{|\beta|+|J|+\tau(J,I\setminus J)}\binom{\alpha}{\beta}t^\beta\xi_J\cdot t^{\alpha-\beta}\pxi{l_s}(\xi_{I\setminus J})\partial\\
&=&\sum_{\substack{0\leqslant\beta\leqslant\alpha\\l_s\in J\subset I}}(-1)^{|\beta|+s+|J\setminus\{l_s\}|+\tau(J\setminus\{l_s\},I\setminus J)}\binom{\alpha}{\beta}t^\beta\xi_{J\setminus\{l_s\}}\cdot t^{\alpha-\beta}\xi_{I\setminus J}\partial\\
&&+\sum_{\substack{0\leqslant\beta\leqslant\alpha\\l_s\notin J\subset I}}(-1)^{|\beta|+s+1+|J|+\tau(J,I\setminus(J\cup\{l_s\}))}\binom{\alpha}{\beta}t^\beta\xi_J\cdot t^{\alpha-\beta}\xi_{I\setminus(J\cup\{l_s\})}\partial\\
&=&(-1)^sX_{\alpha,I\setminus\{l_s\},\partial}+(-1)^{s+1}X_{\alpha,I\setminus\{l_s\},\partial}\\
&=&0.
\end{eqnarray*}

Therefore, $[X_{\alpha,I,\partial},\Delta]=0$.

Now let $|\alpha|+|I|>0$ and $a\in A$. We have
\begin{eqnarray*}
[X_{\alpha,I,\partial},a]&=&\sum_{\substack{0\leqslant \beta\leqslant \alpha\\J\subset I}}(-1)^{|\beta|+|J|+\tau(J,I\setminus J)}\binom{\alpha}{\beta}t^\beta\xi_J\cdot t^{\alpha-\beta}\xi_{I\setminus J}\partial(a)\\
&=&\sum_{\substack{0\leqslant\beta\leqslant\alpha\\J\subset I}}(-1)^{|\beta|+|J|}\binom{\alpha}{\beta}t^\alpha\xi_I\partial(a)\\
&=&(1-1)^{|\alpha|+|I|}t^\alpha\xi_I\partial(a)\\
&=&0.
\end{eqnarray*}

Hence, $[T,A]=[T,\Delta]=0$.

(2)For any $\vn\neq K\subset I$, $\xi_K\xi_{I\setminus K}\neq 0$ and
\begin{eqnarray*}
&&\sum_{J\subset K}(-1)^{\tau(J,I\setminus J)+|K\setminus J|+\tau(K\setminus J,I\setminus K)+\tau(J,K\setminus J)}\xi_K\xi_{I\setminus K}\\
&=&\sum_{J\subset K}(-1)^{\tau(J,I\setminus J)+|K\setminus J|+\tau(K\setminus J,I\setminus K)}\xi_J\xi_{K\setminus J}\xi_{I\setminus K}\\
&=&\sum_{J\subset K}(-1)^{\tau(J,I\setminus J)+|K\setminus J|}\xi_J\xi_{I\setminus J}\\
&=&\sum_{J\subset K}(-1)^{|K\setminus J|}\xi_I\\
&=&0.
\end{eqnarray*}
So $\sum_{J\subset K}(-1)^{\tau(J,I\setminus J)+|K\setminus J|+\tau(K\setminus J,I\setminus K)+\tau(J,K\setminus J)}=0$.

Then we have
\begin{eqnarray*}
&&\sum_{\substack{0\leqslant\beta\leqslant\alpha\\J\subset I}}(-1)^{\tau(J,I\setminus J)}\binom{\alpha}{\beta}t^\beta\xi_J\cdot X_{\alpha-\beta,I\setminus J,\partial}\\
&=&\sum_{\substack{0\leqslant\beta\leqslant\alpha\\J\subset I}}\sum_{\substack{0\leqslant\beta'\leqslant\alpha-\beta\\J'\subset I\setminus J}}(-1)^{\tau(J,I\setminus J)+|\beta'|+|J'|+\tau(J',I\setminus(J\cup J'))}\binom{\alpha}{\beta}\binom{\alpha-\beta}{\beta'}\\
&&t^\beta\xi_J\cdot t^{\beta'}\xi_{J'}\cdot t^{\alpha-\beta-\beta'}\xi_{I\setminus(J\cup J')}\partial\\
&=&\sum_{\substack{0\leqslant\beta\leqslant\alpha\\K\subset I}}\sum_{\substack{0\leqslant\beta'\leqslant\alpha-\beta\\J\subset K}}(-1)^{\tau(J,I\setminus J)+|\beta'|+|K\setminus J|+\tau(K\setminus J,I\setminus K)+\tau(J,K\setminus J)}\binom{\alpha}{\beta+\beta'}\binom{\beta+\beta'}{\beta}\\
&&t^{\beta+\beta'}\xi_K\cdot t^{\alpha-\beta-\beta'}\xi_{I\setminus K}\partial\\
&=&\sum_{\substack{K\subset I\\J\subset K}}\sum_{\substack{0\leqslant\gamma\leqslant\alpha\\0\leqslant\beta\leqslant\gamma}}(-1)^{\tau(J,I\setminus J)+|\gamma-\beta|+|K\setminus J|+\tau(K\setminus J,I\setminus K)+\tau(J,K\setminus J)}\binom{\alpha}{\gamma}\binom{\gamma}{\beta}t^\gamma\xi_K\cdot t^{\alpha-\gamma}\xi_{I\setminus K}\partial\\
&=&\sum_{\substack{K\subset I\\J\subset K}}(-1)^{\tau(J,I\setminus J)+|K\setminus J|+\tau(K\setminus J,I\setminus K)+\tau(J,K\setminus J)}\sum_{0\leqslant\gamma\leqslant\alpha}\binom{\alpha}{\gamma}(1-1)^{|\gamma|}(-1)^{|\gamma|}t^\gamma\xi_K\cdot t^{\alpha-\gamma}\xi_{I\setminus K}\partial\\
&=&t^\alpha\xi_I\partial.
\end{eqnarray*}
\end{proof}

\begin{lemma}\label{Talg}
\begin{itemize}
  \item[(1)]$\mathcal{B}=\{X_{\alpha,I,\partial}\ |\ \alpha\in\Z_+^m,I\subset\{1,\dots,n\},\partial\in\{\frac{\partial}{\partial t_1},\dots,\frac{\partial}{\partial t_m},\pxi{1},\dots,\pxi{n}\}\}$ is an $A$-basis of the free left $A$-module $A\cdot W$.
  \item[(2)]$T=\{x\in A\cdot W\ |\ [x,A]=[x,\Delta]=0\}$. Thus $T$ is a Lie subalgebra of $A\cdot W$.
\end{itemize}
\end{lemma}
\begin{proof}
(1)By Lemma \ref{compute} (2), $\mathcal{B}$ is a generating set of the free left $A$-module $A\cdot W$. And it is easy to see that $\mathcal{B}$ is $A$-linearly independent.

(2)Let $T_1=\{x\in A\cdot W\ |\ [x,A]=[x,\Delta]=0\}$. Then $T\subset T_1$ by Lemma \ref{compute} (1). Let $x\in T_1$, from (1) we know that $x=\sum_{i=1}^ka_i\cdot x_i+x'$, where $x_1,\dots,x_k$ are linearly independent elements in $T$, $a_1,\dots,a_k\in A$ and $x'\in A\cdot \Delta$. For any $a\in A$, $0=[x,a]=[x',a]$. So $x'=0$. For any $y\in \Delta$, $0=[y,x]=\sum_{i=1}^ky(a_i)\cdot x_i$. So $a_1,\dots,a_k\in\C$. Therefore $T_1\subset T$ and consequently $T=T_1$.
\end{proof}

Let $\mathcal{K}_{m,n}$ be the associative subalgebra of $\bar U$ generated by $A$ and $\Delta$.

\begin{lemma}\label{Ksimple}
Any simple $\mathcal K_{m,n}$-module is strictly simple.
\end{lemma}
\begin{proof}
Let $V$ be a simple $\mathcal K_{m,n}$-module and $V'$ be a nonzero $\mathcal K_{m,n}$-invariant subspace of $V$ with $v\in V'\setminus\{0\}$. Since $\pxi{1},\dots,\pxi{n}$ act nilpotently on the finite-dimensional subspace $\C[\pxi{1},\dots,\pxi{n}]v$ of $V'$, there is a $w\in\C[\pxi{1},\dots,\pxi{n}]v\setminus\{0\}$ such that $\pxi{i}\cdot w=0,i=1,\dots,n$. If $w$ is homogeneous, $\mathcal K_{m,n}w$ is a submodule of $V$. By the simplicity of $V$, $\mathcal K_{m,n}w=V'=V$.

Suppose that $w$ is not homogeneous. Then $w=w_0+w_1$ with $w_0\in V_{\bar 0}\setminus\{0\}$ and $w_1\in V_{\bar 1}\setminus\{0\}$. Clearly, $\pxi{i}\cdot w_0=0,i=1,\dots,n$. Also, $\mathcal K_{m,n}w_0=V$. Consequently, there is an odd element $x$ in the subalgebra of $\mathcal K_{m,n}$ that is generated by $t_1,\dots,t_m,\xi_1,\dots,\xi_n,\frac{\partial}{\partial t_1},\dots,\frac{\partial}{\partial t_m}$, such that $xw_0=w_1$. Note that $x^{n+1}=0$. Let $k$ be the smallest positive integer such that $x^kw_0\neq 0$ and $x^{k+1}w_0=0$. Then $x^kw=x^kw_0+x^{k+1}w_0=x^kw_0\in V'$ is a nonzero homogeneous element in $V'$. By the simplicity of $V$, $\mathcal K_{m,n}x^kw_0=V'=V$.

Therefore $V$ is strictly simple.
\end{proof}

Note that $\mathcal K_{m,n}\cong  \mathcal K_{(1)}\otimes\mathcal K_{(2)}\otimes\cdots\otimes\mathcal K_{(m+n)}$, where $\mathcal K_{(i)}$ is the subalgebra of $\mathcal K_{m,n}$ generated by $t_i,\frac{\partial}{\partial t_i}$ for $i\in\{1,\dots,m\}$, and $\mathcal K_{(m+j)}$ is the subalgebra of $\mathcal K_{m,n}$ generated by $\xi_j,\frac{\partial}{\partial \xi_j}$ for $j\in\{1,\dots,n\}$. 

From \cite{FGM} and Lemma 3.5 in \cite{XL1}, we have
\begin{lemma}\label{Kweight}

(1) Let $P$ be any simple weight $\mathcal{K}_{m,n}$-module. Then $P\cong V_1\otimes \cdots \otimes V_m\otimes \C[\xi_1]\otimes\cdots\otimes \C[\xi_n]$, where every $V_i$ is one of the following simple weight $\C[t_i,\frac{\partial}{\partial t_i}]$-modules: $$t_i^{\lambda_1}\C[t_i^{\pm 1}],\C[t_i],\C[t_i^{\pm 1}]/\C[t_i],$$ where $\lambda_i\in \C\setminus \Z$.

(2) Any weight $\mathcal{K}_{m,n}$-module $V$ must have a simple submodule $V'$. Moreover, $\supp(V')=X_1\times\dots\times X_m\times S$, where each $X_i\in\{\lambda_i+\Z,\Z_+,-\N\}$ for some $\lambda_i\in\C\setminus\Z$ and $S=\{k_1e_{1}+\dots+k_ne_{n}\ |\ k_1,\dots,k_n=0,1\}$.
\end{lemma}

\begin{lemma}
There is an associative superalgebra isomorphism
$$\pi_1:\mathcal{K}_{m,n}\otimes U(T)\rightarrow \bar U,\pi_1(x\otimes y)=x\cdot y,\forall x\in \mathcal{K}_{m,n},y\in U(T).$$
\end{lemma}
\begin{proof}
Since $T$ is a Lie super-subalgebra of $\bar U$, the restriction of $\pi_1$ on $U(T)$ is well-defined. By Lemma \ref{compute} (1), $\pi_1(\mathcal{K}_{m,n})$ and $\pi_1(U(T))$ are super-commutative. So $\pi_1$ is a well-defined homomorphism of associative superalgebras.

Let $\g=A\otimes T+(A\Delta+A)\otimes\C$. From Lemma \ref{Talg} (1), it is easy to see that $\iota:=\pi_1|_{\g}:\g\rightarrow A\cdot W+A$ is a Lie superalgebra isomorphism. Then the restriction of $\iota^{-1}$ on $\tilde{W}=W+A$ gives a Lie superalgebra homomorphism $\eta:\tilde{W}\rightarrow\mathcal{K}_{m,n}\otimes U(T)$. By Lemma \ref{compute} (2),
\begin{eqnarray*}
&\eta(t^\alpha\xi_I)=t^\alpha\xi_I\otimes 1,\eta(t^\alpha\xi_I\partial)=\sum_{\substack{0\leqslant \beta\leqslant \alpha\\J\subset I\\ \alpha\neq\beta\ or\ J\neq I}}(-1)^{\tau(J,I\setminus J)}\binom{\alpha}{\beta}t^\beta\xi_J\otimes X_{\alpha-\beta,I\setminus J,\partial}+t^\alpha\xi_I\partial\otimes 1,\\
&\forall \alpha\in\Z_+^m,I\subset\{1,\dots,n\},\partial\in\{\frac{\partial}{\partial t_1},\dots,\frac{\partial}{\partial t_m},\pxi{1},\dots,\pxi{n}\}.
\end{eqnarray*}

$\eta$ induces an associative superalgebra homomorphism $\tilde \eta:U(\tilde W)\rightarrow \mathcal{K}_{m,n}\otimes U(T)$. Clearly, $\tilde\eta(\mathcal{J})=0$. Then we have the induced associative superalgebra homomorphism $\bar\eta:\bar U\rightarrow \mathcal{K}_{m,n}\otimes U(T)$. It's easy to see that $\pi_1=\bar\eta^{-1}$. Hence $\pi_1$ is an isomorphism.
\end{proof}

Let $\m$ be the ideal of $A$ generated by $t_1,\dots,t_m,\xi_1,\dots,\xi_n$. Then $\m\Delta$ is a super-subalgebra of $W$. For any $k\in\N$, $\m^k\Delta$ is an ideal of $\m\Delta$.

\begin{lemma}
The linear map $\pi_2:\m\Delta\rightarrow T$ defined by
$$\pi_2(t^\alpha\xi_I\partial)=X_{\alpha,I,\partial},\forall t^\alpha\xi_I\in\m,\partial\in\{\frac{\partial}{\partial t_1},\dots,\frac{\partial}{\partial t_m},\pxi{1},\dots,\pxi{n}\},$$
is an isomorphism of Lie superalgebras.
\end{lemma}
\begin{proof}
$\pi_2$ is clearly an isomorphism of superspaces. Consider the following combination of natural Lie superalgebra homomorphisms:
$$\m\Delta\subset \m\cdot\Delta+A\cdot T\rightarrow(\m\cdot \Delta+A\cdot T)/(\m\cdot\Delta+\m\cdot T)\rightarrow(A\cdot T)/(\m\cdot T)\rightarrow T$$
This homomorphism, which maps $t^\alpha\xi_I\partial$ to $X_{\alpha,I,\partial}$, is just the linear map $\pi_2$.
\end{proof}

\begin{lemma}\label{pi3}
\begin{itemize}
  \item[(1)]$\m\Delta/\m^2\Delta\cong\gl(m,n)$.
  \item[(2)]Suppose that $V$ is a simple weight $\m\Delta$-module. Then $\m^2\Delta\cdot V=0$. Thus $V$ could be viewed as a simple weight $\gl(m,n)$-module.
\end{itemize}
\end{lemma}
\begin{proof}
(1)Define the linear map $\pi_3:\gl(m,n)\rightarrow\m\Delta/\m^2\Delta$ by
\begin{eqnarray*}
&\pi_3(E_{i,j})=t_i\partial_j+\m^2\Delta,\pi_3(E_{m+p,j})=\xi_p\partial_j+\m^2\Delta,\pi_3(E_{i,m+q})=t_i\pxi{q}+\m^2\Delta,\\
&\pi_3(E_{m+p,m+q})=\xi_p\pxi{q}+\m^2\Delta,\forall i,j\in\{1,\dots,m\},p,q\in\{1,\dots,n\}.
\end{eqnarray*}
It is easy to verify that $\pi_3$ is an isomorphism of Lie superalgebras.

(2)It is clear that the adjoint action of $d=d_1+\dots+d_m+\delta_1+\dots+\delta_n$ on $W$ is diagonalizable. For any $k\in\Z$, let $W_{k}=\{x\in W\ |\ [d,x]=kx\}$. Then
\begin{eqnarray*}
&W_{k-1}=\span\{t^\alpha\xi_I\frac{\partial}{\partial t_i},t^\alpha\xi_I\pxi{j}\ |\ \alpha\in\Z_+^m,I\subset\{1,\dots,n\},\\
&|\alpha|+|I|=k,i\in\{1,\dots,m\},j\in\{1,\dots,n\}\},\forall k\in\Z_+
\end{eqnarray*}
and $$W=\sum_{l=-1}^\infty W_l,\m^k\Delta=\sum_{l=k-1}^\infty W_l,\forall k\in\N.$$

Let $v$ be a nonzero homogeneous weight vector of $V$. Then $d\cdot v=cv$ for some $c\in\C$. By the simplicity of $V$, $V=U(\m\Delta)v$. So the action of $d$ on $V$ is diagonalizable and the eigenvalues of $d$ on $V$ are contained in $c+\Z_+$. Since $\m^2\Delta$ is an ideal of $\m\Delta$, $\m^2\Delta V$ is a submodule of $V$. The eigenvalues of $d$ on $\m^2\Delta V$ are contained in $c+\N$, which does not contain $c$. Thus $\m^2\Delta V\neq V$. By the simplicity of $V$, $\m^2\Delta V=0$.
\end{proof}

We now have the associative superalgebra homomorphism $\pi:\bar U\stackrel{\pi_1^{-1}}{\longrightarrow}\mathcal K_{m,n}\otimes U(T)\stackrel{1\otimes \pi_2^{-1}}{\longrightarrow}\mathcal K_{m,n}\otimes U(\m\Delta)\longrightarrow\mathcal K_{m,n}\otimes U(\m\Delta/\m^2\Delta)\stackrel{1\otimes\pi_3^{-1}}{\longrightarrow}\mathcal K_{m,n}\otimes U(\gl(m,n))$ with
\begin{eqnarray*}
&\pi(t^\alpha\xi_I\frac{\partial}{\partial t_i})=t^\alpha\xi_I\frac{\partial}{\partial t_i}\otimes 1+\sum_{k=1}^m\alpha_kt^{\alpha-e_k}\xi_I\otimes E_{k,i}+(-1)^{|I|-1}\sum_{k=1}^n\pxi{k}(t^\alpha\xi_I)\otimes E_{m+k,i},\\
&\pi(t^\alpha\xi_I\pxi{j})=t^\alpha\xi_I\pxi{j}\otimes 1+\sum_{k=1}^m\alpha_kt^{\alpha-e_k}\xi_I\otimes E_{k,m+j}\\
&+(-1)^{|I|-1}\sum_{k=1}^n\pxi{k}(t^\alpha\xi_I)\otimes E_{m+k,m+j},\\
&\pi(t^\alpha\xi_I)=t^\alpha\xi_I\otimes 1,\forall \alpha\in\Z_+^m,I\subset\{1,\dots,n\},i\in\{1,\dots,m\},j\in\{1,\dots,n\}.
\end{eqnarray*}

Let $P$ be a $\mathcal K_{m,n}$-module and $M$ be a $\gl(m,n)$-module. Define the $AW$-module $F(P,M)=P\otimes M$ by
$$x\cdot(u\otimes v)=\pi(x)\cdot(u\otimes v),x\in\bar U,u\in P,v\in M.$$
Note that
$$\pi(d_i)=d_i\otimes 1+1\otimes E_{i,i},\pi(\delta_j)=\delta_j\otimes 1+1\otimes E_{m+j,m+j}.$$
Then $F(P,M)$ is a weight $AW$-module if $P$ is a weight $\mathcal K_{m,n}$-module and $M$ is a weight $\gl(m,n)$-module.

\begin{lemma}\label{cusF(P,M)}
Suppose that $P$ is a simple weight $\mathcal K_{m,n}$-module and $M$ is a simple weight $\gl(m,n)$-module. Then $F(P,M)$ is a bounded $AW$-module if and only if $M$ is isomorphic to the $\gl(m,n)$-module $L(V_1\otimes V_2)$ for some finite-dimensional simple $\gl_m$-module $V_1$ and some simple bounded $\gl_n$-module $V_2$ up to a parity-change.
\end{lemma}
\begin{proof}
Clearly, $P$ is a simple weight $\mathcal K_{m,n}$-module with one-dimensional weight spaces. By Lemma \ref{Kweight}, $\supp(P)$ is of the form $X_1\times\dots\times X_m\times S$, where each $X_i\in\{\lambda^{(0)}_i+\Z,\Z_+,-\N\}$ for some $\lambda^{(0)}_i\in\C\setminus\Z$ and $S=\{k_1e_{1}+\dots+k_ne_{n}\ |\ k_1,\dots,k_n=0,1\}$.

Let $V_1$ be a finite-dimensional simple $\gl_m$-module and $V_2$ be a simple bounded $\gl_n$-module. Then there is a positive integer $N$ such that
$$\dim(V_2)_\mu\leqslant N,\forall\mu\in\supp(V_2).$$
The $\gl(m,n)$-module $K(V_1\otimes V_2)=\Lambda(\gl(m,n)_{-1})\otimes(V_1\otimes V_2)$. Note that $\Lambda(\gl(m,n)_{-1})\otimes V_1$ is a finite-dimensional weight module over the standard Cantan subalgebra of $\gl(m,n)$. Then $K(V_1\otimes V_2)$ has weight spaces with dimensions no more than
$$\dim \Lambda(\gl(m,n)_{-1})\cdot \dim V_1 \cdot N=2^{mn}N \dim V_1.$$
Moreover, $\supp(K(V_1\otimes V_2))\subset S_1\times (\mu+\Z^n)$ for some finite set $S_1\subset \C^m$ and some $\mu\in\supp(V_2)$. So $L(V_1\otimes V_2)$ has weight spaces with bounded municipality and $\supp(L(V_1\otimes V_2))\subset S_1\times (\mu+\Z^n)$.

For any given $(\lambda',\mu')\in\supp(F(P,L(V_1\otimes V_2)))$, there are at most $|S||S_1|$ pairs $(\lambda^{(1)},\mu^{(1)})\in X_1\times\dots\times X_m\times S,(\lambda^{(2)},\mu^{(2)})\in S_1\times (\mu+\Z^n)$ such that $(\lambda^{(1)},\mu^{(1)})+(\lambda^{(2)},\mu^{(2)})=(\lambda',\mu')$. Since any weight space of $P$ is one-dimensional and $L(V_1\otimes V_2)$ is bounded, $F(P,L(V_1\otimes V_2))$ is bounded. Clearly, $F(P,\Pi(L(V_1\otimes V_2)))$ is also bounded. Thus, the sufficiency is proved.

Now suppose that $F(P,M)$ is a bounded $AW$-module. By Lemma \ref{L(V)}, there is a simple weight $\gl(m,n)_0$-module $V$ such that $M$ is isomorphic to the $\gl(m,n)$-module $L(V)$ or $\Pi(L(V))$. Note that $F(P,\Pi(L(V)))\cong F(\Pi(P),L(V))$, where $\Pi(P)$ is also a simple weight $\mathcal K_{m,n}$-module. Without loss of generality, we assume that $M=L(V)$. If $m=0$, $V$ is a simple bounded $\gl_n$-module and $M=V\cong L(V_1,V)$, where $V_1$ is the one-dimensional trivial $\gl_0$-module. Suppose $m>0$.

Let $v$ be a nonzero weight vector of weight $(\lambda',\mu')$ in $V\subset L(V)$. Then $U(\gl_m)v$ is a weight $\gl_m$-submodule of $V$ with
$$E_{m+i,m+i}w=\mu'_iw,\forall i\in\{1,\dots,n\},w\in U(\gl_m)v.$$

We claim that $U(\gl_m)v$ is finite-dimensional. It's easy to see that any weight space of $U(\gl_m)v$ is finite-dimensional. For any positive integer $k$, suppose that $U(\gl_m)v$ has $k$ pairwise different weights $\lambda'+\beta^{(1)},\dots,\lambda'+\beta^{(k)}$ for some $\beta^{(1)},\dots,\beta^{(k)}\in\Z^m$. Let $v_1,\dots,v_k$ be nonzero weight vectors in $U(\gl_m)v$ of weight $\lambda'+\beta^{(1)},\dots,\lambda'+\beta^{(k)}$ respectively. Let $\tilde\lambda\in\C^m$ such that $\tilde\lambda-\beta^{(i)}\in X_1\times\dots\times X_m$ for all $i\in\{1,\dots,k\}$ and let $\tilde\mu\in S$. Then $(\tilde\lambda-\beta_1,\tilde\mu),\dots,(\tilde\lambda-\beta_k,\tilde\mu)$ are pairwise different weights of $P$. Let $w_1,\dots,w_k$ be nonzero weight vectors in $P$ of weight $(\tilde\lambda-\beta_1,\tilde\mu),\dots,(\tilde\lambda-\beta_k,\tilde\mu)$ respectively. Obviously, $w_1\otimes v_1,\dots,w_k\otimes v_k$ are weight vectors in $F(P,M)$ of weight $(\tilde\lambda+\lambda',\tilde\mu+\mu')$. So $\dim(F(P,M))_{(\tilde\lambda+\lambda',\tilde\mu+\mu')}\geqslant k$. Since $F(P,M)$ is a bounded $AW$-module, $U(\gl_m)v$ has only finite weights. Therefore, $U(\gl_m)v$ is finite-dimensional.

Let $V_1$ be a simple $\gl_m$-submodule of $U(\gl_m)v$, which is clearly also a finite-dimensional weight module. By Lemma \ref{density}(2), there is a simple $\gl_n$-module $V_2$ such that $V\cong V_1\otimes V_2$. $V_2$ is a weight module for $V$ is, and $V_2$ is a bounded module for $F(P,L(V))$ is. Thus, $M\cong L(V_1\otimes V_2)$ and the necessity is proved.
\end{proof}

\begin{theorem}\label{F(P,M)}
Suppose that $V$ is a simple weight $AW$-module with $\dim V_{(\lambda,\mu)}<\infty$ for some $(\lambda,\mu)\in\supp(V)$. Then $V$ is isomorphic to $F(P,M)$ for some simple weight $\mathcal K_{m,n}$-module $P$ and some simple weight $\gl(m,n)$-module $M$.
\end{theorem}
\begin{proof}
Let $D=\span\{X_{e_i,\vn,\frac{\partial}{\partial t_i}},X_{0,\xi_j,\pxi{j}},t_i\cdot \frac{\partial}{\partial t_i},\xi_j\cdot\pxi{j}\ |\ i=1,\dots,m,j=1,\dots,n\}\subset\bar U$. Then $D$ is an abelian Lie super-subalgebra of $\bar U$ and $D\cdot V_{(\lambda,\mu)}\subset V_{(\lambda,\mu)}$. So $V_{(\lambda,\mu)}$ contains a homogeneous common eigenvector $v$ of $D$. Let $\rho:\mathcal K_{m,n}\otimes U(\m\Delta)\stackrel{1\otimes\pi_2}{\longrightarrow}\mathcal K_{m,n}\otimes U(T)\stackrel{\pi_1}{\longrightarrow}\bar U$ be the isomorphism of associative superalgebras. Then $V$ could be viewed as a simple $\mathcal K_{m,n}\otimes U(\m\Delta)$-module via $\rho$ and $v$ is a common eigenvector of $\rho^{-1}(D)=\span\{d_i\otimes 1,\delta_j\otimes 1,1\otimes d_i,1\otimes \delta_j\ |\ i=1,\dots,m,j=1,\dots,n\}$. It follows that $\mathcal K_{m,n}v$ is a weight $\mathcal K_{m,n}$-module. By Lemma \ref{Kweight}, $\mathcal K_{m,n}v$ has a simple submodule $P$. Then, from Lemma \ref{Ksimple} and Lemma \ref{density} (2), $V\cong P\otimes M$ for some simple $U(\m\Delta)$-module $M$. Since $\rho^{-1}(D)$ acts diagonally on $(\mathcal K_{m,n}\otimes U(\m\Delta))\cdot v=V$, $M$ is a simple weight $\m\Delta$-module. By Lemma \ref{pi3}, $M$ could be viewed as a simple weight $\gl(m,n)$-module. So $V$ is isomorphic to the simple $AW$-module $F(P,M)$.
\end{proof}

\section{Classification of Bounded Modules}

If $m\in\N$, for any
$$\alpha,\beta\in\Z_+^m,I,J\subset\{1,\dots,n\},r\in\N,j\in\{1,\dots,m\}, \partial,\partial'\in\{\frac{\partial}{\partial t_1},\dots,\frac{\partial}{\partial t_m},\frac{\partial}{\partial\xi_1},\dots,\frac{\partial}{\partial\xi_n}\},$$
define
\begin{equation}\omega_{\alpha,\beta,I,J}^{r,j,\partial,\partial'}=\sum_{i=0}^r(-1)^i\binom{r}{i}t^{\alpha+(r-i)e_j}\xi_I\partial\cdot t^{\beta+ie_j}\xi_J\partial'\in U(W).
\end{equation}
Then
\begin{equation}
\omega_{\alpha+e_j,\beta,I,J}^{r,j,\partial,\partial'}-\omega_{\alpha,\beta+e_j,I,J}^{r,j,\partial,\partial'}=\omega_{\alpha,\beta,I,J}^{r+1,j,\partial,\partial'}.
\end{equation}

\begin{lemma}\cite[Lemma 4.5]{XL1}\label{Wn+omega}
Suppose that $m\in\N$ and $M$ is a weight $W_{m,0}$-module with $\dim M_\lambda\leqslant N$ for all $\lambda\in \supp(M)$. Then there is an $r\in \N$ such that $\omega_{\alpha,\beta,\varnothing,\varnothing}^{r,j,\frac{\partial}{\partial t_i},\frac{\partial}{\partial t_i'}}\cdot M=0$ for all $\alpha,\beta\in\Z_+^m,i,i',j\in\{1,\dots,m\}$.
\end{lemma}

\begin{lemma}\label{omega}
Suppose that $m\in\N$ and $M$ is a bounded $W$-module. Then there is an $r\in\N$ such that $\omega_{\alpha,\beta,I,J}^{r,j,\partial,\partial'}\cdot M=0$ for all
$$\alpha,\beta\in\Z_+^m,I,J\subset\{1,\dots,n\},j\in\{1,\dots,m\}, \partial,\partial'\in\{\frac{\partial}{\partial t_1},\dots,\frac{\partial}{\partial t_m},\frac{\partial}{\partial\xi_1},\dots,\frac{\partial}{\partial\xi_n}\}.$$
\end{lemma}
\begin{proof}
Since $M$ is bounded, there is an $N\in\N$ such that $\dim M_{(\lambda,\mu)}\leqslant N$ for all $(\lambda,\mu)\in\supp(M)$. For any $(\lambda,\mu)\in\supp(M)$, $M_{((\lambda+\Z^m),\mu)}:=\oplus_{\alpha\in\Z^m}M_{((\lambda+\alpha),\mu)}$ is a weight $W_{m,0}$-module, the dimensions of whose weight spaces are uniformly bounded by $N$. By Lemma \ref{Wn+omega}, there is an $r\in \N$ such that $\omega_{\alpha,\beta,\varnothing,\varnothing}^{r,j,\frac{\partial}{\partial t_i},\frac{\partial}{\partial t_i'}}\cdot M_{((\lambda+\alpha),\mu)}=0$ for all $\alpha,\beta\in\Z_+^m,i,i',j\in\{1,\dots,m\}$. It follows that
$$\omega_{\alpha,\beta,\varnothing,\varnothing}^{r,j,\frac{\partial}{\partial t_i},\frac{\partial}{\partial t_i'}}\cdot M=0,\forall \alpha,\beta\in\Z_+^m,i,i',j\in\{1,\dots,m\}.$$

For any $\alpha,\beta,\gamma\in\Z_+^m,I\subset\{1,\dots,n\},j\in\{1,\dots,m\}$, we have
\begin{equation}\begin{split}
&[\omega_{\alpha,\beta,\varnothing,\varnothing}^{r,j,\frac{\partial}{\partial t_j},\frac{\partial}{\partial t_j}},t^\gamma\xi_I\frac{\partial}{\partial t_j}]\\
=&\sum_{i=0}^r(-1)^i\binom{r}{i}[t^{\alpha+(r-i)e_j}\frac{\partial}{\partial t_j},t^\gamma\xi_I\frac{\partial}{\partial t_j}]\cdot t^{\beta+ie_j}\frac{\partial}{\partial t_j}\\
&+\sum_{i=0}^r(-1)^i\binom{r}{i}t^{\alpha+(r-i)e_j}\frac{\partial}{\partial t_j}\cdot[t^{\beta+ie_j}\frac{\partial}{\partial t_j},t^\gamma\xi_I\frac{\partial}{\partial t_j}]\\
=&\sum_{i=0}^r(-1)^i\binom{r}{i}(\gamma_jt^{\alpha+\gamma+(r-i-1)e_j}\xi_I\frac{\partial}{\partial t_j}-(\alpha_j+r-i)t^{\alpha+\gamma+(r-i-1)e_j}\xi_I\frac{\partial}{\partial t_j})\cdot t^{\beta+ie_j}\frac{\partial}{\partial t_j}\\
&+\sum_{i=0}^r(-1)^i\binom{r}{i}t^{\alpha+(r-i)e_j}\frac{\partial}{\partial t_j}\cdot(\gamma_jt^{\beta+\gamma+(i-1)e_j}\xi_I\frac{\partial}{\partial t_j}-(\beta_j+i)t^{\beta+\gamma+(i-1)e_j}\xi_I\frac{\partial}{\partial t_j}).
\end{split}\end{equation}
Then
\begin{equation}\begin{split}
&f(\alpha,\beta,\gamma):
=[\omega_{\alpha+e_j,\beta,\varnothing,\varnothing}^{r,j,\frac{\partial}{\partial t_j},\frac{\partial}{\partial t_j}},t^\gamma\xi_I\frac{\partial}{\partial t_j}]
-[\omega_{\alpha,\beta,\varnothing,\varnothing}^{r,j,\frac{\partial}{\partial t_j},\frac{\partial}{\partial t_j}},t^{\gamma+e_j}\xi_I\frac{\partial}{\partial t_j}]\\
=&-2\sum_{i=0}^r(-1)^i\binom{r}{i}t^{\alpha+\gamma+(r-i)e_j}\xi_I\frac{\partial}{\partial t_j}\cdot t^{\beta+ie_j}\frac{\partial}{\partial t_j}\\
&+\sum_{i=0}^r(-1)^i\binom{r}{i}t^{\alpha+(r-i+1)e_j}\frac{\partial}{\partial t_j}\cdot(\gamma_jt^{\beta+\gamma+(i-1)e_j}\xi_I\frac{\partial}{\partial t_j}-(\beta_j+i)t^{\beta+\gamma+(i-1)e_j}\xi_I\frac{\partial}{\partial t_j})\\
&-\sum_{i=0}^r(-1)^i\binom{r}{i}t^{\alpha+(r-i)e_j}\frac{\partial}{\partial t_j}\cdot((\gamma_j+1)t^{\beta+\gamma+ie_j}\xi_I\frac{\partial}{\partial t_j}-(\beta_j+i)t^{\beta+\gamma+ie_j}\xi_I\frac{\partial}{\partial t_j})\in\ann(M).
\end{split}\end{equation}
We have
\begin{equation}\begin{split}
&f(\alpha+e_j,\beta,\gamma)-f(\alpha,\beta,\gamma+e_j)\\
=&\sum_{i=0}^r(-1)^i\binom{r}{i}t^{\alpha+(r-i+2)e_j}\frac{\partial}{\partial t_j}\cdot(\gamma_j-\beta_j-i)t^{\beta+\gamma+(i-1)e_j}\xi_I\frac{\partial}{\partial t_j}\\
&-2\sum_{i=0}^r(-1)^i\binom{r}{i}t^{\alpha+(r-i+1)e_j}\frac{\partial}{\partial t_j}\cdot(\gamma_j-\beta_j-i+1)t^{\beta+\gamma+ie_j}\xi_I\frac{\partial}{\partial t_j}\\
&+\sum_{i=0}^r(-1)^i\binom{r}{i}t^{\alpha+(r-i)e_j}\frac{\partial}{\partial t_j}\cdot(\gamma_j-\beta_j-i+2)t^{\beta+\gamma+(i+1)e_j}\xi_I\frac{\partial}{\partial t_j}.
\end{split}\end{equation}
Then
\begin{equation}\begin{split}
&(f(\alpha+e_j,\beta+e_j,\gamma)-f(\alpha,\beta+e_j,\gamma+e_j))-(f(\alpha+e_j,\beta,\gamma+e_j)-f(\alpha,\beta,\gamma+2e_j))\\
=&-2\omega_{\alpha+2e_j,\beta+\gamma,\varnothing,I}^{r,j,\frac{\partial}{\partial t_j},\frac{\partial}{\partial t_j}}+4\omega_{\alpha+e_j,\beta+\gamma+e_j,\varnothing,I}^{r,j,\frac{\partial}{\partial t_j},\frac{\partial}{\partial t_j}} -2\omega_{\alpha,\beta+\gamma+2e_j,\vn,I}^{r,j,\frac{\partial}{\partial t_j},\frac{\partial}{\partial t_j}}\\
=&-2\omega_{\alpha+e_j,\beta+\gamma,\vn,I}^{r+1,j,\frac{\partial}{\partial t_j},\frac{\partial}{\partial t_j}}+2\omega_{\alpha,\beta+\gamma+e_j,\vn,I}^{r+1,j,\frac{\partial}{\partial t_j},\frac{\partial}{\partial t_j}}\\
=&-2\omega_{\alpha,\beta+\gamma,\vn,I}^{r+2,j,\frac{\partial}{\partial t_j},\frac{\partial}{\partial t_j}}\in\ann(M).
\end{split}\end{equation}
So
\begin{equation}\label{omega1}\omega_{\alpha,\beta,\vn,I}^{r+2,j,\frac{\partial}{\partial t_j},\frac{\partial}{\partial t_j}}\in\ann(M)\end{equation}
for all $\alpha,\beta\in\Z_+^m,I\subset\{1,\dots,n\},j\in\{1,\dots,m\}$.

Let $J\subset\{1,\dots,n\}$. We have
\begin{equation}\begin{split}
&[\omega_{\alpha,\beta,\vn,I}^{r+2,j,\frac{\partial}{\partial t_j},\frac{\partial}{\partial t_j}},t^\gamma\xi_J\frac{\partial}{\partial t_j}]\\
=&\sum_{i=0}^{r+2}(-1)^i\binom{r+2}{i}(-1)^{|I||J|}[t^{\alpha+(r+2-i)e_j}\frac{\partial}{\partial t_j},t^\gamma\xi_J\frac{\partial}{\partial t_j}]\cdot t^{\beta+ie_j}\xi_I\frac{\partial}{\partial t_j}\\
&+\sum_{i=0}^{r+2}(-1)^i\binom{r+2}{i}t^{\alpha+(r+2-i)e_j}\frac{\partial}{\partial t_j}\cdot[t^{\beta+ie_j}\xi_I\frac{\partial}{\partial t_j},t^\gamma\xi_J\frac{\partial}{\partial t_j}]\\
=&\sum_{i=0}^{r+2}(-1)^i\binom{r+2}{i}(-1)^{|I||J|}(\gamma_j-\alpha_j-r-2+i)t^{\alpha+\gamma+(r+1-i)e_j}\xi_J\frac{\partial}{\partial t_j}\cdot t^{\beta+ie_j}\xi_I\frac{\partial}{\partial t_j}\\
&+\sum_{i=0}^{r+2}(-1)^i\binom{r+2}{i}(\gamma_j-\beta_j-i)t^{\alpha+(r+2-i)e_j}\frac{\partial}{\partial t_j}\cdot t^{\beta+\gamma+(i-1)e_j}\xi_I\xi_J\frac{\partial}{\partial t_j}.
\end{split}\end{equation}
From (\ref{omega1}) we have
$$h(\alpha,\beta,\gamma):=\sum_{i=0}^{r+2}(-1)^i\binom{r+2}{i}(\gamma_j-\alpha_j-r-2+i)t^{\alpha+\gamma+(r+1-i)e_j}\xi_J\frac{\partial}{\partial t_j}\cdot t^{\beta+ie_j}\xi_I\frac{\partial}{\partial t_j}\in\ann(M).$$
Then
$$h(\alpha+e_j,\beta,\gamma)-h(\alpha,\beta,\gamma+e_j)=-2\omega_{\alpha+\gamma,\beta,J,I}^{r+2,j,\frac{\partial}{\partial t_j},\frac{\partial}{\partial t_j}}\in\ann(M).$$
So
\begin{equation}\label{omega2}
\omega_{\alpha,\beta,J,I}^{r+2,j,\frac{\partial}{\partial t_j},\frac{\partial}{\partial t_j}}\in\ann(M)
\end{equation}
for all $\alpha,\beta\in\Z_+^m,I,J\subset\{1,\dots,n\},j\in\{1,\dots,m\}$.

Let $\partial\in\{\frac{\partial}{\partial t_1},\dots,\frac{\partial}{\partial t_m},\frac{\partial}{\partial\xi_1},\dots,\frac{\partial}{\partial\xi_n}\}\setminus\{\frac{\partial}{\partial t_j}\}$. We have
\begin{equation}\begin{split}
&[\omega_{\alpha,\beta,J,I}^{r+2,j,\frac{\partial}{\partial t_j},\frac{\partial}{\partial t_j}},t^\gamma\partial]\\
=&\sum_{i=0}^{r+2}(-1)^i\binom{r+2}{i}(-1)^{|I||\partial|}(\gamma_jt^{\alpha+\gamma+(r+1-i)e_j}\xi_J\partial-(-1)^{|J||\partial|}t^{\gamma+(r+2-i)e_j}\partial(t^\alpha\xi_J)\frac{\partial}{\partial t_j})\cdot t^{\beta+ie_j}\xi_I\frac{\partial}{\partial t_j}\\
&+\sum_{i=0}^{r+2}(-1)^i\binom{r+2}{i}t^{\alpha+(r+2-i)e_j}\xi_J\frac{\partial}{\partial t_j}\cdot(\gamma_jt^{\beta+\gamma+(i-1)e_j}\xi_I\partial
-(-1)^{|I||\partial|}t^{\gamma+ie_j}\partial(t^\beta\xi_I)\frac{\partial}{\partial t_j}).
\end{split}\end{equation}
By (\ref{omega2}),
\begin{equation}\begin{split}
&x(\alpha,\beta,\gamma):=
\sum_{i=0}^{r+2}(-1)^i\binom{r+2}{i}(-1)^{|I||\partial|}\gamma_jt^{\alpha+\gamma+(r+1-i)e_j}\xi_J\partial\cdot t^{\beta+ie_j}\xi_I\frac{\partial}{\partial t_j}\\
&+\sum_{i=0}^{r+2}(-1)^i\binom{r+2}{i}\gamma_jt^{\alpha+(r+2-i)e_j}\xi_J\frac{\partial}{\partial t_j}\cdot t^{\beta+\gamma+(i-1)e_j}\xi_I\partial\in\ann(M).\\
\end{split}\end{equation}
Then
\begin{equation}\begin{split}
&x(\alpha+e_j,\beta,\gamma)-x(\alpha,\beta,\gamma+e_j)\\
=&-\sum_{i=0}^{r+2}(-1)^i\binom{r+2}{i}(-1)^{|I||\partial|}t^{\alpha+\gamma+(r+2-i)e_j}\xi_J\partial\cdot t^{\beta+ie_j}\xi_I\frac{\partial}{\partial t_j}\\
&+\sum_{i=0}^{r+2}(-1)^i\binom{r+2}{i}\gamma_jt^{\alpha+(r+3-i)e_j}\xi_J\frac{\partial}{\partial t_j}\cdot t^{\beta+\gamma+(i-1)e_j}\xi_I\partial\\
&-\sum_{i=0}^{r+2}(-1)^i\binom{r+2}{i}(\gamma_j+1)t^{\alpha+(r+2-i)e_j}\xi_J\frac{\partial}{\partial t_j}\cdot t^{\beta+\gamma+ie_j}\xi_I\partial.
\end{split}\end{equation}
Thus,
\begin{equation}\begin{split}
&y(\alpha,\beta,\gamma):=(x(\alpha+2e_j,\beta,\gamma)-x(\alpha+e_j,\beta,\gamma+e_j))\\
&-(x(\alpha+e_j,\beta,\gamma+e_j)-x(\alpha,\beta,\gamma+2e_j))\\
=&\sum_{i=0}^{r+2}(-1)^i\binom{r+2}{i}\gamma_jt^{\alpha+(r+4-i)e_j}\xi_J\frac{\partial}{\partial t_j}\cdot t^{\beta+\gamma+(i-1)e_j}\xi_I\partial\\
&-2\sum_{i=0}^{r+2}(-1)^i\binom{r+2}{i}(\gamma_j+1)t^{\alpha+(r+3-i)e_j}\xi_J\frac{\partial}{\partial t_j}\cdot t^{\beta+\gamma+ie_j}\xi_I\partial\\
&+\sum_{i=0}^{r+2}(-1)^i\binom{r+2}{i}(\gamma_j+2)t^{\alpha+(r+2-i)e_j}\xi_J\frac{\partial}{\partial t_j}\cdot t^{\beta+\gamma+(i+1)e_j}\xi_I\partial\in\ann(M).\\
\end{split}\end{equation}
We have
\begin{equation}\begin{split}
&y(\alpha,\beta+e_j,\gamma)-y(\alpha,\beta,\gamma+e_j)\\
=&-\omega_{\alpha+2e_j,\beta+\gamma,J,I}^{r+2,j,\frac{\partial}{\partial t_j},\partial}+2\omega_{\alpha+e_j,\beta+\gamma+e_j,J,I}^{r+2,j,\frac{\partial}{\partial t_j},\partial}-\omega_{\alpha,\beta+\gamma+2e_j,J,I}^{r+2,j,\frac{\partial}{\partial t_j},\partial}\\
=&-\omega_{\alpha+e_j,\beta+\gamma,J,I}^{r+3,j,\frac{\partial}{\partial t_j},\partial}+\omega_{\alpha,\beta+\gamma+e_j,J,I}^{r+3,j,\frac{\partial}{\partial t_j},\partial}\\
=&-\omega_{\alpha,\beta+\gamma,J,I}^{r+4,j,\frac{\partial}{\partial t_j},\partial}\in\ann(M).
\end{split}\end{equation}
So
\begin{equation}\label{omega3}
\omega_{\alpha,\beta,J,I}^{r+4,j,\frac{\partial}{\partial t_j},\partial}\in\ann(M)
\end{equation}
for all $\alpha,\beta\in\Z_+^m,I,J\subset\{1,\dots,n\},j\in\{1,\dots,m\},\partial\in\{\frac{\partial}{\partial t_1},\dots,\frac{\partial}{\partial t_m},\frac{\partial}{\partial\xi_1},\dots,\frac{\partial}{\partial\xi_n}\}\setminus\{\frac{\partial}{\partial t_j}\}$.
Similarly,
\begin{equation}\label{omega4}
\omega_{\alpha,\beta,J,I}^{r+4,j,\partial,\frac{\partial}{\partial t_j}}\in\ann(M).
\end{equation}

Let $\partial'\in\{\frac{\partial}{\partial t_1},\dots,\frac{\partial}{\partial t_m},\frac{\partial}{\partial\xi_1},\dots,\frac{\partial}{\partial\xi_n}\}\setminus\{\frac{\partial}{\partial t_j}\}$. We have
\begin{equation}\begin{split}
&[\omega_{\alpha,\beta,J,I}^{r+4,j,\frac{\partial}{\partial t_j},\partial},t^\gamma\partial']\\
=&\sum_{i=0}^{r+4}(-1)^i\binom{r+4}{i}(-1)^{|\xi_I\partial||\partial'|}(\gamma_jt^{\alpha+\gamma+(r+3-i)e_j}\xi_J\partial'\\
&-(-1)^{|J||\partial'|}t^{\gamma+(r+4-i)e_j}\partial'(t^\alpha\xi_J)\frac{\partial}{\partial t_j})\cdot t^{\beta+ie_j}\xi_I\partial\\
&+\sum_{i=0}^{r+4}(-1)^i\binom{r+4}{i}t^{\alpha+(r+4-i)e_j}\xi_J\frac{\partial}{\partial t_j}\cdot(t^{\beta+ie_j}\partial(t^\gamma\xi_I)\partial'-(-1)^{|\xi_I\partial||\partial'|}t^{\gamma+ie_j}\partial'(t^\beta\xi_I)\partial).
\end{split}\end{equation}
By (\ref{omega3}),
$$\sum_{i=0}^{r+4}(-1)^i\binom{r+4}{i}\gamma_jt^{\alpha+\gamma+(r+3-i)e_j}\xi_J\partial'\cdot t^{\beta+ie_j}\xi_I\partial\in\ann(M).$$
Taking $\gamma=e_j$, we have
\begin{equation}\label{omega5}
\omega_{\alpha,\beta,J,I}^{r+4,j,\partial',\partial}\in\ann(M)
\end{equation}
for all $\alpha,\beta\in\Z_+^m,I,J\subset\{1,\dots,n\},j\in\{1,\dots,m\},\partial,\partial'\in\{\frac{\partial}{\partial t_1},\dots,\frac{\partial}{\partial t_m},\frac{\partial}{\partial\xi_1},\dots,\frac{\partial}{\partial\xi_n}\}\setminus\{\frac{\partial}{\partial t_j}\}$.

The lemma follows from (\ref{omega2}), (\ref{omega3}), (\ref{omega4}) and (\ref{omega5}) after replacing $r+4$ with $r$.
\end{proof}

Let $V$ be a weight $W$-module. Then $W\otimes V$ is a weight $W$-module with
$$(W\otimes V)_{(\lambda,\mu)}=\oplus_{\alpha\in\Z^m,\beta\in\Z^n}(W)_{(\alpha,\beta)}\otimes V_{((\lambda-\alpha),(\mu-\beta))},\forall \lambda\in\C^m,\mu\in\C^n.$$
Define the action of $A$ on $W\otimes V$ by
$$a\cdot(x\otimes v)=(ax)\otimes v,\forall a\in A,x\in W,v\in V.$$
It is easy to verify that $W\otimes V$ now is an $AW$-module. Define a linear map $\theta:W\otimes V\rightarrow V$ by
$$\theta(x\otimes v)=x\cdot v,\forall x\in W,v\in V.$$
Then $\theta$ is a $W$-module homomorphism. Let $X(V)=\{x\in\rm{Ker} \theta\ |\ A\cdot x\subset \rm{Ker}\theta\}$. Clearly, $X(V)$ is an $AW$-submodule of $W\otimes V$. The $AW$-module $\hat V=(W\otimes V)/X(V)$ is called the $A$-cover of $V$ and $\theta$ induces a $W$-module homomorphism $\hat \theta:\hat V\rightarrow V$.

\begin{theorem}\label{hat}
Suppose that $V$ is a simple bounded $W$-module. Then $\hat V$ is a bounded $AW$-module.
\end{theorem}
\begin{proof}
This is obvious if $V$ is trivial. Suppose that $V$ is not trivial, then $W\cdot V=V$ by the simplicity of $V$.

If $m=0$, $W$ is finite-dimensional. Then the $AW$-module $W\otimes V$ is bounded for $V$ is. So $\hat V$ is bounded.

Suppose $m\in\N$. Let $(\lambda,\mu)\in\supp(V)$, then $\supp(V)\subset(\lambda,\mu)+\Z^{m+n}$.
Since $V$ is bounded, there is a $k\in\N$ such that
$$V_{(\lambda',\mu')}\leqslant k,\forall (\lambda',\mu')\in\supp(V),$$
and there is an $r\in\N$ such that $\omega_{\alpha,\beta,I,J}^{r,j,\partial,\partial'}\cdot V=0$ for all $\alpha,\beta\in\Z_+^m,I,J\subset\{1,\dots,n\},j\in\{1,\dots,m\}, \partial,\partial'\in\{\frac{\partial}{\partial t_1},\dots,\frac{\partial}{\partial t_m},\frac{\partial}{\partial\xi_1},\dots,\frac{\partial}{\partial\xi_n}\}$ by Lemma \ref{omega}. Then we have
\begin{equation}\label{3.2}
\sum_{i=0}^r(-1)^i\binom{r}{i}t^{\alpha+(r-i)e_j}\xi_I\partial\otimes t^{\beta+ie_j}\xi_J\partial'v\in X(V),\forall v\in V.
\end{equation}

Let
\begin{eqnarray*}
B=\span\{t^\alpha\xi_I\partial\ |\ \alpha\in\Z_+^m,\alpha_i\leqslant r,i=1,\dots,m,\\
I\subset\{1,\dots,n\},\partial\in\{\frac{\partial}{\partial t_1},\dots,\frac{\partial}{\partial t_m},\frac{\partial}{\partial\xi_1},\dots,\frac{\partial}{\partial\xi_n}\}\}
\end{eqnarray*}
be a subspace of $W$. Then $B$ is a finite-dimensional $H$-submodule of $W$. So $B\otimes V$ is a $H$-submodule of $W\otimes V$ with
$$\dim (B\otimes V)_{(\lambda',\mu')}\leqslant \dim B\cdot k,\forall \lambda'\in\lambda+\Z^m,\mu'\in\mu+\Z^n.$$

To prove that $\hat V$ is a bounded $AW$-module, we need only to prove that $W\otimes V=B\otimes V+X(V)$. Since $W\cdot V=V$, it suffices to prove that
\begin{eqnarray*}
&t^\alpha\xi_I\partial\otimes t^\beta\xi_J\partial'v\in B\otimes V+X(V),\forall \alpha,\beta\in\Z_+^m,I,J\subset\{1,\dots,n\},\\
&\partial,\partial'\in\{\frac{\partial}{\partial t_1},\dots,\frac{\partial}{\partial t_m},\frac{\partial}{\partial\xi_1},\dots,\frac{\partial}{\partial\xi_n}\},v\in V.
\end{eqnarray*}
We will prove this by induction on $|\alpha|=\alpha_1+\dots+\alpha_m$. This is clear if $|\alpha|\leqslant r$ or $\alpha_j\leqslant r$ for all $j\in\{1,\dots,m\}$. Suppose that $\alpha_j\geqslant r$ for some $j$. By (\ref{3.2}), $\sum_{i=0}^r(-1)^i\binom{r}{i}t^{(\alpha-re_j)+(r-i)e_j}\xi_I\partial\otimes t^{\beta+ie_j}\xi_J\partial'v\in X(V)$. So
$$t^\alpha\xi_I\partial\otimes t^\beta\xi_J\partial'v\in-\sum_{i=1}^r(-1)^i\binom{r}{i}t^{(\alpha-re_j)+(r-i)e_j}\xi_I\partial\otimes t^{\beta+ie_j}\xi_J\partial'v+X(V).$$
Note that $|(\alpha-re_j)+(r-i)e_j|<|\alpha|$ for all $i\geqslant 1$. By the induction hypothesis, the right hand side is contained in $B\otimes V+X(V)$. Thus $t^\alpha\xi_I\partial\otimes t^\beta\xi_J\partial'v\in B\otimes V+X(V)$.
\end{proof}

\begin{theorem}\label{main}
Let $V$ be a simple non-trivial bounded $W$-module. Then $V$ is a simple $W$-quotient of the $AW$-module $F(P,L(V_1\otimes V_2))$, where $P$ is a simple weight $\mathcal{K}_{m,n}$-module, $V_1$ is a finite-dimensional simple $\gl_m$-module and $V_2$ is a simple bounded $\gl_n$-module.
\end{theorem}
\begin{proof}
As defined above, $\hat\theta:\hat V\rightarrow V$ is a $W$-module homomorphism, so $\hat\theta$ maps any $AW$-submodule of $\hat V$ to $0$ or $V$ by the simplicity of $V$.
In particular, since $V$ is non-trivial, $\hat\theta(\hat V)=W\cdot V=V$. Let $(\lambda,\mu)\in\supp(V)$, then $(\lambda,\mu)\in\supp(\hat V)$ and $\hat\theta(\hat V_{(\lambda,\mu)})=V_{(\lambda,\mu)}$. By Theorem \ref{hat}, $\hat V$ is bounded. So $\hat V_{(\lambda,\mu)}$ is finite-dimensional.

Let $\bar U_0=\{x\in\bar U\ |\ [x,H]=0\}$. Then $\bar U_0$ is a subalgebra of $\bar U$ and $\hat V_{(\lambda,\mu)}$ is a $\bar U_0$-module. Let $M$ be a minimal $\bar U_0$-submodule of $\hat V_{(\lambda,\mu)}$ such that $\hat\theta(M)\neq 0$ and $M'$ be a maximal $\bar U_0$-submodule of $M$. Since $(\bar UM)_{(\lambda,\mu)}=M$ and $(\bar UM')_{(\lambda,\mu)}=M'$, we have $\hat\theta(\bar UM)=V$ and $\hat\theta(\bar UM')=0$. It follows that $V$ is a $W$-quotient of the simple bounded $AW$-module $\bar UM/\bar UM'$.

From Theorem \ref{F(P,M)} and Lemma \ref{cusF(P,M)}, $\bar UM/\bar UM'$ is isomorphic to $F(P,L(V_1\otimes V_2))$, where $P$ is a simple weight $\mathcal{K}_{m,n}$-module, $V_1$ is a finite-dimensional simple $\gl_m$-module and $V_2$ is a simple bounded $\gl_n$-module. So $V$ is a simple quotient of $F(P,L(V_1\otimes V_2))$.
\end{proof}

We remark that simple weight $\mathcal{K}_{m,n}$-modules and simple bounded weight $gl_n$ modules are known, see Lemma \ref{Kweight} in this paper and Theorem 13.3 in \cite{Ma1}. 

{\bf Ackowledgement.} {This work is partially supported by NSF of China (Grant 11471233, 11771122, 11971440)}.

 \noindent R.L\"u: Department of Mathematics, Soochow University, Suzhou, P. R. China.  Email: rlu@suda.edu.cn

\vspace{0.2cm}\noindent Y. Xue.: Department of Mathematics, Soochow University, Suzhou, P. R. China.  Email: yhxue00@stu.suda.edu.cn, corresponding author

\end{document}